\documentclass{article}
\usepackage{amsmath,amssymb,mathrsfs,amsthm}
\usepackage{array,enumerate}
\usepackage{latexsym}
\usepackage{mathrsfs}
\usepackage{amscd}
\usepackage[all]{xy}
\newtheorem{thm}{Theorem}[section]
\newtheorem{dfn}[thm]{Definition}
\newtheorem{prop}[thm]{Proposition}
\newtheorem{cor}[thm]{Corollary}
\newtheorem{lem}[thm]{Lemma}
\newtheorem{rem}[thm]{Remark}
\newtheorem{cnst}[thm]{Construction}
\title{Special values of zeta functions of varieties over finite fields via higher Chow groups}
\author{Hiroyasu Miyazaki}

\begin{document}

\maketitle

\begin{abstract}
In this paper we study special values of zeta functions of singular varieties over finite fields.
We give a new formula of special values by constructing a morphism of homology theories, which we call regulator, from Bloch's higher Chow group to weight homology.
The main idea of our proof is to use weight spectral sequence of homology theories, whose $E^{1}$-terms are homology groups for smooth projective schemes.
\end{abstract}

\section*{Introduction}
\addcontentsline{toc}{section}{Introduction}

For a separeted scheme $X$ of finite type, the zeta function of $X$ is defined to be an infinite product
\[
\zeta_{X}(s) := \prod_{x \in X : closed}\frac{1}{1-(\# k(x))^{-s}},
\]
where the product runs over all closed points of $X$, $k(x)$ denotes the finite residue field for a closed point $x$ and $\#k(x)$ its cardinality.
Note that the infinite product converges absolutely for any complex number $s$ satisfying $Re(s) \gg 0$ (\cite{Se}). 

If $X$ is a separated scheme of finite type over a finite field $k=\mathbb{F}_{q}$ of $q$ elements, then the zeta function of $X$ has another expression
\[
\zeta_{X}(s) = \prod_{x \in X : closed}\frac{1}{1-t^{\mathrm{deg}(x)}}
\]
by regarding $t=q^{-s}$ as a new variable.
Here, $\deg(x)$ is the degree of finite field extension $k(x)/k$.
We denote the right hand side of the above equation by $Z_{X/k}(t)$, or simply by $Z_{X}(t)$ if the choice of base field is clear from the context.
As a consequence of Grothendieck's determinant formula, $Z_{X}(t)$ is a rational function with rational coefficients.
For an integer $r$, the special value $\zeta_{X}(r)^{\ast}$ of $\zeta_{X}(s)$ at $s=r$ is the leading coefficient of Laurent expansion of $Z_{X}(t)$ at $t=q^{-r}$.
Note that any special value is always a rational number.

\ 

There are a lot of remarkable formulas which express special values via arithmetic invariants.
Milne's formula which uses \'{e}tale cohomology is a famous example (\cite{Mil}). 
Geisser found similar formulas for Weil-\'{e}tale cohomology (\cite{Gei1}) and arithmetic homology (\cite{Gei4}) and so on.
Kerz and Saito proved the following formula, which uses higher Chow groups, in \cite{KeS1}:
for any smooth proper geometrically irreducible scheme $X$ of dimension $d$ over a finite field, 
\[
\zeta_{X}(0)^{\ast} = \prod_{i=0}^{2d} (\# \mathrm{CH}_{0}(X,i)_{\mathrm{tor}})^{(-1)^{i}}.
\]
Our main goal is to generalize this formula to the case that $X$ is an arbitrary singular variety over a finite field.

\ 

In this paper, we use two invariants to express our formula. One is  Bloch's higher Chow groups $\mathrm{CH}_{0}(X,i)$, and the other is the weight homology $H_{i}^{W}(X)$, which was introduced by Gillet-Soul\'{e} and Jannsen in the situation that resolution of singularities exsists (\cite{GS1}, \cite{Jan}, \cite{Jan2}).
Kerz-Saito and Gillet-Soul\'{e} extended their construction over any perfect field by using Gabber's refined alteration instead of resolution of singularities (\cite{Il}, \cite{KeS2}, \cite{GS2}).
These can be regarded as homology theories (Definition \ref{def_homology_theory}) on the category $\mathrm{Var}_{k}$ of varieties over a perfect field $k$.
In the second section, we construct a morphism of homology theories
\[
\mathrm{Reg}_{\ast} : \mathrm{CH}_{0}(-,\ast) \otimes_{\mathbb{Z}} \Lambda \to H_{\ast}^{W},
\]
which we call "regulator", where $\Lambda = \mathbb{Z}$ if resolution of singularities holds, and $\Lambda = \mathbb{Z}[1/p]$ otherwise ($p$ is the characteristic of the base field).
Before stating the main result, let us consider the followng condition for any quasi-projective scheme $X$ of dimension $d$ over a finite field $k$:

\ 

$( \spadesuit )_{X/k}$: 
the regulator map $\mathrm{Reg}_{i}(X)_{\mathbb{Q}} : \mathrm{CH}_{0}(X,i)_{\mathbb{Q}} \to H_{i}^{W}(X)_{\mathbb{Q}}$ after tensoring $\mathbb{Q}$ over $\mathbb{Z}$ is surjective for each $i = 2, \dots, d$.

\ 

Assuming Parshin's conjecture, which implies that higher Chow groups of smooth proper schemes over a finite fields is torsion,  $( \spadesuit )_{X/k}$ holds for any $X$.
Indeed, one can prove that $\mathrm{Reg}_{i}(X)_{\mathbb{Q}}$ is an {\it isomorphism} (see Remark \ref{parshin_bijectivity}).

\begin{thm}\label{main_thm_resol_intro}(Theorem \ref{main_thm_bounded})
Let $k$ be a finite field $\mathbb{F}_{q}$ with characteristic $p$, and $X$ be a quasi-projective scheme of dimension $d$ over $k$.
Suppose that the condition $( \spadesuit )_{X/k}$ holds.
Then, $\chi (\mathrm{Reg}_{i}(X))$ is well-defined as a positive rational number for each $i$, equals to $1$  for each $i>2d$, and the following equation holds up to sign and a power of $p$:
\[
\zeta_{X}(0)^{\ast} = \prod_{i=0}^{2d} \chi (\mathrm{Reg}_{i}(X))^{(-1)^{i+1}}.
\]
Moreover, if resolution of singularities exists for any  variety of dimension $\leq d$ over $k$, then the equation holds only up to sign.
\end{thm}
Here, $\chi(\mathrm{Reg}_{i}) := \#(\mathrm{Cok}(\mathrm{Reg}_{i})_{\mathrm{tor}})/\#(\mathrm{Ker}(\mathrm{Reg}_{i})_{\mathrm{tor}})$. 
Since resolution of singularities exists for any variety of dimension smaller than or equal to 3, we have the equation only up to sign for such varieties (\cite{Abh}, \cite{CP1}, \cite{CP2}).

\

By Theorem \ref{main_thm_resol_intro}, we obtain a result on special values at $s<0$.
Note that we do not need surjectivity condition in this case:

\begin{cor} (Corollary \ref{s<0_bounded})
Let $k$ be a finite field $\mathbb{F}_{q}$, $p=ch(k)$ and $X$ be a quasi-projective scheme of dimension $d$ over $k$.
Then, for any negative integer $r$, $\chi (\mathrm{Reg}_{i}(X \times \mathbb{A}^{-r}))$ is well-defined for each $i$, equals to $1$  for each $i>2(d-r)$, and the following equations hold up to sign and a power of $p$:
\[
\zeta_{X}(r)^{\ast} = \prod_{i=0}^{2(d-r)} \chi (\mathrm{Reg}_{i}(X \times \mathbb{A}^{-r}))^{(-1)^{i+1}}
=\prod_{i=0}^{2(d-r)} (\# \mathrm{CH}_{r}(X,i)_{\mathrm{tor}})^{(-1)^{i}}.
\]
Moreover, if resolution of singularities exists for any  variety of dimension $\leq d$ over $k$, then the equation holds only up to sign.
\end{cor}

To understand the meaning of weight homology, let us consider an example: a smooth scheme $U$ which has a smooth compactification $U \to X$ such that the closed complement $Y=X \setminus U$ with reduced structure is a simple normal crossing divisor on $X$.
Moreover, denote by $(Y_{i})_{0 \leq i \leq N}$ an ordered set of all irreducible components of $Y$.
Since $Y$ is a simple normal crossing divisor, the disjoint union of irreducible components
\[
Y^{(a)} := \coprod_{i_{1} < \cdots < i_{a}} Y_{i_{0}} \cap \cdots \cap Y_{i_{a}}
\]
is a projective smooth scheme over $k$ for any non-negative integer $a$.
(Here, we write $Y^{(0)} := X$.)
Then the weight homology of $U$ can be calculated as the homology of the following complex:
\[
\cdots \to \Lambda^{\pi_{0}(Y^{(a+1)})} \to \Lambda^{\pi_{0}(Y^{(a)})} \to \cdots
\to \Lambda^{\pi_{0}(Y^{(1)})} \to \Lambda^{\pi_{0}(Y^{(0)})} \to 0.
\]
Each differential homomorphism
$\Lambda^{\pi_{0}(Y^{(a+1)})} \to \Lambda^{\pi_{0}(Y^{(a)})}$
is an alternating sum of the maps induced by inclusions $Y^{(a+1)} \to Y^{(a)}$.
This example shows that for non-projective smooth schemes, weight homology is an invariant which reflects the combinatoric information of the complement in the fixed compactification.

\

The regulator map $\mathrm{Reg}_{\ast} : \mathrm{CH}_{0}(X,\ast) \otimes \Lambda \to H_{\ast}^{W}(X)$  is in degree $0$ the degree homomorphism of the Chow group of zero cycles.
In fact, the restriction of the regulator $Reg|_{\mathrm{SP}_{k}}$ to the subcategory $\mathrm{SP}_{k} \subset \mathrm{Var}_{k}$ consisting of smooth projective varieties can be expressed as follows:
for any projective smooth scheme $X$ over $k$,
\[
\mathrm{Reg}_{i}(X) = \begin{cases}
		\deg_{X/k} : \mathrm{CH}_{0}(X,i) \otimes \Lambda \to \Lambda^{\pi_{0}(X)}, & (i=0)\\
		0 : \mathrm{CH}_{i}(X,i) \otimes \Lambda \to 0. & (i>0)
	\end{cases}
\]
It can be seen that the kernel of $\deg_{X/k}$ is equal to the torsion subgroup $\mathrm{CH}_{0}(X)_{\mathrm{tor}}$ of $\mathrm{CH}_{0}(X)$.

\ 

The key ingredient of the proof of Theorem \ref{main_thm_resol_intro} is a homological spectral sequence called "weight spectral sequence" for homology theories (see Remark \ref{rem_sncd}).
The $E^{1}$-terms of a weight spectral sequence consist of homology groups for smooth projective schemes.
In fact, regulator maps can be regarded as a part of a morphism between two weight spectral sequences corresponding to higher Chow groups and weight homology groups, respectively.
Thus, one can reduce the statement of Theorem \ref{main_thm_resol_intro} to the case that the scheme $X$
is smooth projective.
However, note that to calculate the alternating product in the theorem, one needs to prove that only  finitely many $E^{1}$-terms in the spectral sequences appear in the computation (otherwise, the alternating product is not well-defined).
This problem is solved by the following theorem (see Theorem \ref{boundedness}):

\begin{thm}
The weight complex (also in the meaning of Kerz-Saito) is contained in the homotopy category of bounded complexes of effective Chow motives.
\end{thm}

If one assumes that resolution of singularities exists, the above theorem was already proved by Gillet-Soul\'{e}, hence the new contribution of mine is the case of positive characteristics. 
The proof needs some arguments comparing two kinds of weight complex functors (defined by Gillet-Soul\'{e}-Kerz-Saito and Bondarko respectively) using Kelly's motivic resolution theorem (see Theorem \ref{boundedness}).

\

Remark: in the case that $X$ is a smooth geometrically irreducible surface, Theorem \ref{main_thm_resol_intro} was essentially proved by Kato (\cite[Proposition 7.3]{K}).
Kato used homology of Milnor-Gersten complex instead of Bloch's higher Chow group to state the result, but it is known that these give the same invariant for surfaces.
On the other hand, he also constructed an invariant corresponding to weight homology and defined regulator.

Inspired by Kato's result, the author came up with the idea to reconstruct the regulator map as a morphism of homology theories.
There are two merits of using homology theory method.
One of them is that this method makes it possible to construct regulators for higher dimensional varieties.
Though Kato's construction is elementary, it is not clear that it can be extended to the case $\dim(X) \geq 3$.
The other is that any homology theory has localization exact sequences, which allow us to use induction on dimension of varieties.
By this, we can prove the first main result not only for smooth varieties but also for any singular one.
Moreover, using Gabber's refined alterations, we can drop the assumption on resolution of singularities.

\ 

Now we will list the contents of the paper.
In the first section we recall definition of homology theory, and some methods to construct homology theories.
Most of them are just the summary of the results of \cite{GS1}, \cite{Jan2} and \cite{KeS1}.

The second section is dedicated to preparing the ingredients of the statements of our main results.
We prove that Bloch's higher Chow groups provides a homology theory, construct weight homology and define regulator maps.

In the third section, we prove our main results on special values of zeta functions.
By using weight spectral sequences of homology theories (and by some modifications on the condition of geometrical irreducibility) we can deduce the results from the case that varieties are smooth projective geometrically irreducible.
In this case, the alternating product of $\chi (\mathrm{Reg}_{\ast})$, the Euler characteristics of regulator maps, is just the alternating product of the cardinalities of torsion subgroups of Bloch's higher Chow groups.
In order to show that the alternating product is well-defined, we prove a result on boundedness of weight complexes.
Boundedness of weight complex (in the sense of Gillet-Soul\'{e}) has been known under resolution of singularities, 
but not in positive characteristics, and we need a new method to prove boundedness of weight complexes in that case (see Theorem \ref{boundedness} for details).
The formula of special values via higher Chow groups in the case that the scheme $X$ is smooth projective can be reduced to Milne's famous formula via \'{e}tale cohomology groups, using cycle maps between higher Chow groups and \'{e}tale cohomology groups (cf. \cite{K}, \cite{KeS1}).

\subsection*{Acknowledgements}
I would like to express the deepest gratitude to my supervisors Professor Shuji Saito and Professor Tomohide Terasoma who provided helpful suggestions and strongly encouraged me to keep studying my theme.
Also, this work was supported by the Program for Leading Graduate Schools, MEXT, Japan.

\section{Preliminaries : homology theory and weight complex}

In this section, we recall briefly the definitions of homology theories and the descent construction of them.

\subsection{Notation and convention}

Let $k$ be a field, and $p$ be its characteristic.
For simplicity, we assume that $k$ is perfect. For example, we can take $k=\mathbb{F}_{q}$, a finite field.
Denote by $\Lambda$ either $\mathbb{Z}$ or $\mathbb{Z}[1/p]$.

\ 

Let $\mathrm{Var}_{k}$ be the category of quasi-projective schemes over the field $k$, and $(\mathrm{Var}_{k})_{\ast}$ be its subcategory consisting of the same objects as $\mathrm{Var}_{k}$ and proper morphisms.
Moreover, we denote by $\mathrm{SP}_{k}$ the full subcategory of $\mathrm{Var}_{k}$ of projective smooth schemes over $k$.

By the word "resolution of singularities," we mean the following statement:
for any reduced $X \in \mathrm{Var}_{k}$, there exists a proper birational morphism $f : X^{\prime} \to X$ over $k$ with $X^{\prime}$ smooth over $k$.

\subsection{Homology theory}

\begin{dfn}\label{def_homology_theory}
A homology theory on the category $(\mathrm{Var}_{k})_{\ast}$ is a set of covariant functors
\[
H_{i} : (\mathrm{Var}_{k})_{\ast} \longrightarrow \ (\Lambda\text{-modules}),
\] 
satisfying the following conditions:

(i) For each number $i$, $H_{i}$ is contravariant with respect to open immersions.

(ii) If $i : Y \to X$ is a closed immersion with open complement $j : U \to X$, then there exists a localisation long exact sequence which is functorial with respect to open immersions and proper morphisms:
\[
\cdots \longrightarrow H_{i}(Y) \stackrel{i_{\ast}}{\longrightarrow} H_{i}(X) \stackrel{j^{\ast}}{\longrightarrow} H_{i}(U) \longrightarrow H_{i-1}(Y) \longrightarrow \cdots.
\]
A morphism of homology theories is a set of morphisms of functors which commute with all maps appearing in the definition.
\end{dfn}

Assume that we are given a covariant functor
\[
F : \mathrm{SP}_{k} \longrightarrow (\text{complexes of }\Lambda \text{-modules}).
\]
By taking homology of complexes, we get a set of functors:
\[
H_{i}^{F} : \mathrm{SP}_{k} \longrightarrow (\Lambda \text{-modules}).
\]
Under reasonable conditions, we can get a homology theory on $(\mathrm{Var}_{k})_{\ast}$ extending $H_{\ast}^{F}$.
The following theorem gives a criterion for $H_{i}^{F}$ to extend to a homology theory.

\begin{thm}\label{thm:descent2}
(a)Assume that $F$ preserves finite coproducts, and the homology $H_{\ast}^{F}$ factors through the category of effective Chow motives $\mathrm{Chow}_{k}^{\mathrm{eff}}$.
Then, there exists an extension of $H_{\ast}^{F}$ to a homology theory on $(\mathrm{Var})_{\ast}$.
We denote the homology theory by the same notation $H_{\ast}^{F}$.
Moreover, given a morphism of functors $F \to F^{\prime}$, where $F$ and $F^{\prime}$ satisfy the condition above, then there is a canonical morphism of homology theories $H_{\ast}^{F} \to H_{\ast}^{F^{\prime}}$.

(b)Moreover, assume that the covariant functor $F$ is the restriction to the full subcategory $\mathrm{SP}_{k} \subset (\mathrm{Var}_{k})_{\ast}$ of a covariant functor
\[
\Phi : (\mathrm{Var}_{k})_{\ast} \longrightarrow (\text{complexes of }\Lambda\text{-modules}),
\]
satisfying the following conditions:

(i) $\Phi$ is contravariant with respect to flat morphisms.

(ii) If $i : Y \to X$ is a closed immersion with open complement $j : U \to X$, then the induced composition $\Phi(Y) \stackrel{i_{\ast}}{\longrightarrow} \Phi(X) \stackrel{j^{\ast}}{\longrightarrow} \Phi(U)$ is the zero map, and $\Phi(U)$ is canonically isomorphic to the cone of $i_{\ast}$.

(iii) Given a commutative diagram of schemes consists of flat morphisms and proper morphisms, the square of morphism of complexes induced by the functor $\Phi$ is commutative.

(iv) $\Phi$ preserves finite coproducts.

(v) For a finite flat morphism $f : X \to Y$ of degree $d$(i.e. $f_{\ast}\mathcal{O}_{X}$ is a free $\mathcal{O}_{Y}$ module of rank $d$), the composition $f_{\ast}f^{\ast}$ is multiplication by $d$.

Then, if $F$ factors through $\mathrm{Chow}_{k}^{\mathrm{eff}}$, the homology theory $H_{\ast}^{F}$ on $(\mathrm{Var})_{\ast}$ is canonically isomorphic to the homology of $\Phi$.
\end{thm}

\begin{proof}
For (a), We briefly recall only the construction of the homology theory.
For more details, see \cite{KeS2}.
For any $X \in \mathrm{Var}_{k}$, there exists an open immersion $j : X \to \overline{X}$ to a projective scheme $\overline{X}$ with closed complement $i : Y \to \overline{X}$.
There exists a commutative diagram of simplicial schemes:
\[
\begin{CD}
Y_{\bullet} @> \tilde{i} >> \overline{X}_{\bullet} \\
@V f VV @V g VV \\
Y @> j >> \overline{X},
\end{CD}
\]
where $f$ and $g$ are ($\Lambda$-admissible) hiperenvelopes
(here, $Y$ and $\overline{X}$ are regarded as constant simplicial schemes).
Since $Y$ and $\overline{X}$ are projective, $Y_{\bullet}$ and $\overline{X}_{\bullet}$ are projective as simplicial schemes.
So, we get a morphism of complexes $\Phi(Y_{\bullet}) \stackrel{\tilde{i}}{\longrightarrow} \Phi(\overline{X}_{\bullet})$
(differentials with respect to $\bullet$ is induced by the alternating sum of proper face maps of simplicial schemes).
If we denoby by $\mathrm{tot}C$ the total complex of a homological double complex $C$, we define the homology theory to be the homology of the cone of the induced morphism of complexes $\mathrm{tot}\Phi(Y_{\bullet}) \stackrel{\tilde{i}}{\longrightarrow} \mathrm{tot}\Phi(\overline{X}_{\bullet})$

For (b), see Construction 3.6 and Theorem 4.1 in \cite{KeS2}.
\end{proof}

\begin{rem}
Let $X$ be quasi-projective and $\overline{X}$ its projective compactification with closed complement $Y = \overline{X} \setminus X$.
Let $Y_{\bullet}$ and $\overline{X}_{\bullet}$ be hyperenvelopes of $Y$ and $\overline{X}$ respectively, and $Y_{\bullet} \to \overline{X}_{\bullet}$ a morphism of simplicial scheme compatible with the closed immersion $Y \to \overline{X}$.
Regard $Y_{\bullet}$ and $\overline{X}_{\bullet}$ as homological complexes of effective chow motives as in \cite[Section 2]{GS1}, whose differentials are alternating sums of face maps of simplicial schemes.
Then one can prove that the cone $\mathrm{cone}(Y_{\bullet} \to \overline{X}_{\bullet})$ is independent of the choices of hyperenvelopes up to canonical isomorphism in the homotopy category $K^{+}(\mathrm{Chow}^{\mathrm{eff}}(k)[\frac{1}{p}])$ of effective chow motives.
We denote the class of the complex by $W(X)$, which we call the weight complex of $X$.
This notation will be used in the third section.
\end{rem}

\begin{rem}\label{rem_sncd}
Let $X$, $\overline{X}$, $Y$, $Y_{\bullet}$, $\overline{X}_{\bullet}$ as above.
If $(X_{a})_{a \in \mathbb{Z}}$ is a complex of smooth projective schemes which represents the weight complex $W(X)$, then there exists a convergent homological spectral sequence
\[
H_{b}(X_{a}) \Longrightarrow H_{a+b}(X).
\]
for any homology theory $H_{\ast}$ such that $H_{\ast}$ is constructed by Theorem \ref{thm:descent2}(a).

If $X$ is smooth and $Y$ is a simple normal crossing divisor on $\overline{X}$ with the set of irreducible components $\{ Y_{j} \}_{j \in J}$, let $Y^{(0)} = \overline{X}$, and $Y^{(r)} = \coprod_{j_{0} < \cdots < j_{r}} Y_{j_{0}} \cap \cdots \cap Y_{j_{r}}$.
By the argument of \cite[Proposition 3]{GS1}, the above spectral sequence can be rewritten as follows:
\[
H_{b}(Y^{(a)}) \Longrightarrow H_{a+b}(X)
\]
\end{rem}

\section{Construction of regulator}

In this section, we construct regulator maps connecting Bloch's higher Chow groups and weight homology groups, which we will define below.

\ 

The idea is to regard higher Chow groups as a homology theory in the sense of the previous section.

\subsection{Weight homology theory}

\begin{dfn}
Define a covariant functor
\[
W : \mathrm{SP}_{k} \longrightarrow (\text{complexes of }\Lambda \text{-modules})
\]
by the following formula:
\[
W(X) := \Lambda^{\pi_{0}(X)}[0]
\]
where $\pi_{0}(X)$ denotes the set of connected components of $X$.
Since $\Lambda^{\pi_{0}(X)} \simeq \mathrm{Hom}_{\mathbb{Z}}(\mathrm{CH}^{0}(X), \Lambda)$, the functor $W$ factors through $\mathrm{Chow}_{k}^{\mathrm{eff}}$.
By Theorem \ref{thm:descent2}, we can define a homology theory on $(\mathrm{Var}_{k})_{\ast}$ by extending $W$.
We denote this by $H^{W}$ and call it weight homology theory.
\end{dfn}

\subsection{Descent of higher Chow groups}

\begin{prop}\label{prop:chow_is_hom_theory}
For any non-negative integer $r \geq 0$, let $z_{r}(-, *)$ be the $r$-cycle complex, whose homology groups are defined to be higher Chow groups.
Then, the covariant functor
\[
z_{r}(-,\ast) \otimes_{\mathbb{Z}} \Lambda : (\mathrm{Var}_{k})_{\ast} \longrightarrow (\text{homological complexes of }\Lambda \text{-modules})
\]
satisfies the conditions of Theorem \ref{thm:descent2}(b).
\end{prop}

\begin{proof}
Condition (i) is a fundamental functorial property of the cycle complex.
Condition (ii) is the localisation property proved by Bloch (\cite{Bl}).
Condition (iii) follows immediately from the definition of cycle complexes and the base change formula (the compatibility of proper push-outs and flat pull-backs) of algebraic cycles.
Condition (iv) is trivial, and (v) is the projection formula.
\end{proof}

\ 

\subsection{Regulator}

In the following, we understand the word "complex" as "homological complex."
Let $i$ be an integer, and let 
\[
\tau_{\leq i}: (\text{complexes of } \Lambda \text{-modules}) \to (\text{complexes of } \Lambda \text{-modules})
\]
be the truncation functor defined by
\[
\tau_{\leq i}(M_{\ast}) = [\cdots 0 \to M_{i}/\mathrm{Im}(d_{i+1}) \stackrel{d_{i}}{\longrightarrow} M_{i-1} \stackrel{d_{i-1}}{\longrightarrow} M_{i-2} \stackrel{d_{i-2}}{\longrightarrow} \cdots],
\]
where $M_{\ast}$ is a complex of $\Lambda$-modules.

\begin{cnst}\label{def_regulator_maps}
\upshape

Let
\[
z_{0}(-,\ast) \otimes_{\mathbb{Z}} \Lambda |_{\mathrm{SP}_{k}}
\longrightarrow
\tau_{\leq 0}(z_{0}(-,\ast) \otimes_{\mathbb{Z}} \Lambda)
\simeq
(\mathrm{CH}_{0}(-) \otimes_{\mathbb{Z}} \Lambda)[0]
\]
be the natural transformation of functors induced by truncation.

Let
\[
(\mathrm{CH}_{0}(-) \otimes_{\mathbb{Z}} \Lambda)[0]
\stackrel{\mathrm{deg}}{\longrightarrow}
\Lambda ^{\pi_{0}(-)}[0] = W
\]
be the morphism of functors induced by the degree homomorphisms associated to Chow groups of zero cycles of proper schemes.

Consider the composition of above morphisms:
\[
z_{0}(-,\ast) \otimes_{\mathbb{Z}} \Lambda |_{\mathrm{SP}_{k}}
\longrightarrow
(\mathrm{CH}_{0}(-) \otimes_{\mathbb{Z}} \Lambda)[0]
\longrightarrow
W.
\]
Homologies of both sides of the morphism extends to homology theories on $(\mathrm{Var}_{k})_{\ast}$.
By Theorem \ref{prop:chow_is_hom_theory}, the homology theory associated to the left hand side is $\mathrm{CH}_{0}(-,\ast)$.
So, we get a morphism of homology theories:
\[
\mathrm{Reg}_{\ast}
:
\mathrm{CH}_{0}(-,\ast) \otimes_{\mathbb{Z}} \Lambda
\longrightarrow
H_{\ast}^{W}
\]
which we call the regulator map.
\end{cnst}

\section{Special values of zeta functions at $s=0$}

In this section, we give a formula to calculate $\zeta_{X}(0)^{\ast}$ for any variety $X$ over a finite field, by using regulators defined in the previous section.

\ 

Throughout this section, let $k=\mathbb{F}_{q}$, the finite field with $q$ elements.

\subsection{Finiteness condition}

\begin{dfn}
Let $f : A \to B$ be a homomorphism of abelian groups.
Denote by $\mathrm{FQ}$ the full subcategory of the category of abelian groups consisting of objects isomorphic to groups of the following form:
\[
(\text{finite group}) \oplus (\text{uniquely divisible group}).
\]
Assume that $\mathrm{Ker}(f)$ and $\mathrm{Coker}(f)$ are objects of $\mathrm{FQ}$.
Then, we set
\[
\chi (f) := \# \mathrm{Coker}(f)_{\mathrm{tor}} / \# \mathrm{Ker}(f)_{\mathrm{tor}}.
\]
\end{dfn}

Note that $\chi$ is multiplicative with respect to short exact sequences:

\begin{lem}\label{multiplicativity}

Assume that

\[
\begin{CD}
0 @>>> A @>>> B @>>> C @>>> 0 \\
& & @V f VV @V g VV @V h VV \\
0 @>>> A^{\prime} @>>> B^{\prime} @>>> C^{\prime} @>>> 0
\end{CD}
\]

\ 

is a commutative diagram of abelian groups satisfying that the rows are short exact sequences.
If kernels and cokernels of the vertical arrows $f, h$ are objects of $\mathrm{FQ}$, then the kernel and cokernel of $g$ also belong to $\mathrm{FQ}$, and the following equation holds:
\[
\chi(f) \cdot \chi(g)^{-1} \cdot \chi(h) = 1.
\]
\end{lem}

\begin{proof}
By the snake lemma, the statement is reduced to the next lemma.
\end{proof}

\begin{lem}
The category $\mathrm{FQ}$ is closed under the following operations:

(i)taking kernels and cokernels of morphisms in $\mathrm{FQ}$,

(ii)taking extension of groups.

Moreover, the functor
\[
\mathrm{FQ} \longrightarrow (\text{finite groups})
\]
sending $A$ to $A_{\mathrm{tor}}$ is an exact functor.
\end{lem}

\begin{proof}
The only nontrivial statement is that $\mathrm{FQ}$ is closed under taking extension of groups.
Let $0 \to M_{1} \to M_{2} \to M_{3} \to 0$ be a short exact sequence of abelian groups such that $M_{1}$ and $M_{3}$ belong to $\mathrm{FQ}$.
Without loss of generality, we can assume that $M_{1}$ is finite and $M_{3}$ is uniquely divisible.
An easy exercise of homological algebra shows that this sequence splits, hence the statement follows.
\end{proof}

\begin{prop}\label{FQ}
When $\Lambda \neq \mathbb{Z}$, assume resolution of singularities in $dim \leq d$.
If $X$ is a proper smooth scheme of dimension $d$ over a finite field $\mathbb{F}_{q}$,
then $\mathrm{CH}_{0}(X,i) \otimes \Lambda$ is an object of $\mathrm{FQ}$ for each $i>0$,
and is uniquely divisible for each $i>2\dim (X)$.
Moreover, the torsion subgroup $(\mathrm{CH}_{0}(X) \otimes \Lambda)_{\mathrm{tor}}$ of $\mathrm{CH}_{0}(X) \otimes \Lambda$ is finite.
\end{prop}

\begin{proof}
Without loss of generality, one can assume that $X$ is connected, hence of pure dimension $d=\dim(X)$.
Put $M=\mathrm{CH}_{0}(X) \otimes \Lambda$.
An exercise on homological algebra and the fact that the cycle maps from higher Chow groups to \'{e}tale cohomology groups with finite coefficients are isomorphisms (\cite{KeS1}) shows that the group $M \otimes_{\mathbb{Z}} \mathbb{Q}_{l}/\mathbb{Z}_{l}$ vanishes for any prime number $l \neq p$ (assuming resolution of singularities, this is true also for $l=p$).
The proposition follows easily from this.
Details are left to the reader. 
\end{proof}

\begin{cor}\label{parshin}
If $\Lambda \neq \mathbb{Z}$, assume resolution of singularities in $dim \leq d$.
Then, the following conditions are equivalent for any proper smooth scheme $X$ of dimension $d$ over a finite field and for any positive integer $i$:

(1)
$\mathrm{CH}_{0}(X,i) \otimes_{\mathbb{Z}}\Lambda$ is torsion.

(2)
$\mathrm{CH}_{0}(X,i) \otimes_{\mathbb{Z}}\Lambda$ is finitely generated.

(3)
$\mathrm{CH}_{0}(X,i) \otimes_{\mathbb{Z}}\Lambda$ is finite.
\end{cor}

\subsection{Special values of zeta function via regulator}

\begin{prop}\label{proper_smooth_case}
If $X$ is a proper smooth scheme of dimension $d$ over a finite field $k=\mathbb{F}_{q}$,
then $\chi (\mathrm{Reg}_{i}(X))$ is well-defined for each $i$, and the following equation holds up to sign and $p$-power:
\[
\zeta_{X}(0)^{\ast} = \prod_{i=0}^{2\mathrm{dim}(X)} \chi (\mathrm{Reg}_{i}(X))^{(-1)^{i+1}}.
\]
Assuming resolution of singularities in $dim \leq d$, then the equation holds only up to sign.
\end{prop}

\begin{proof}
First of all, let us consider the case that $X$ is geometrically irreducible over $k$.
In this case, the regulator maps are described as follows (in order to avoid ambiguity, we write $X/L$ instead of $X$ if $X$ is considered as a scheme over a field $L$):
\begin{align*}
\mathrm{Reg}_{0}(X/k) &= \deg_{X/k} : \mathrm{CH}_{0}(X) \otimes \Lambda \to \Lambda \\
\mathrm{Reg}_{i}(X/k) &= 0 : \mathrm{CH}_{0}(X,i) \otimes \Lambda \to 0 \ \ \ (if \ i>0)
\end{align*}
Since $X$ is proper, smooth and geometrically irreducible over $k$, we get the following equation up to sign (and up to $p$-power if $\Lambda \neq \mathbb{Z}$):
\[
\zeta_{X/k}(0)^{\ast} = \prod_{i=0}^{2\dim (X)} \# \mathrm{CH}_{0}(X,i)_{\mathrm{tor}}^{(-1)^{i}}
\]
This equation is proved in \cite[Section 10]{KeS1}.
The main method of the proof is to reduce the equation to Milne's formula (\cite[Theorem 0.4]{Mil}) by comparing higher Chow group with \'{e}tale cohomology via cycle map.

As a result of the previous proposition, we are reduced to the well-known fact that the degree homomorphism $\deg_{X/k} : \mathrm{CH}_{0}(X) \to \mathbb{Z}$ is surjective.

\ 

Now, let $X$ be a smooth proper (not necessarily geometrically irreducible over $k$).
Without loss of generality, one can assume that $X$ is connected.
Since $X$ is reduced, proper and connected, the ring of global sections $K=\Gamma (X, \mathcal{O}_{X})$ is a finite extension field over the base field $k$, and $X$ is geometrically irreducible over $K$.
So, the desired result is reduced to the following lemma.
\end{proof}

\begin{lem}
Let $X$ be a proper smooth scheme over a finite field $k$, and $K$ be a finite extension field over $k$.
Assume that the structure morphism $X \to \mathrm{Spec}k$ factors through the morphism $\mathrm{Spec}K \to \mathrm{Spec}k$ induced by inclusion of fields.
Then, the following equations hold:

\[
\zeta_{X/K}(0)^{\ast} = [K:k]\cdot \zeta_{X/k}(0)^{\ast}
\]
\[
\prod_{i=0}^{2\mathrm{dim}(X)} \chi (\mathrm{Reg}_{i}(X/K))^{(-1)^{i+1}} = [K:k] \cdot \prod_{i=0}^{2\mathrm{dim}(X)} \chi (\mathrm{Reg}_{i}(X/k))^{(-1)^{i+1}},
\]
where $[K:k]$ denotes the degree of field extension $K/k$.
\end{lem}

\begin{proof}
Without loss of generality, one can assume that $X$ is connected geometrically irreducible over $K$.
The second equation is a consequence of the fact that
\[
\# \mathrm{Cok}(\deg_{X/k}) = [K:k] \cdot \# \mathrm{Cok}(\deg_{X/K}).
\]

In order to show the first equation of the lemma, it suffices to note that a functional equation
\[
Z_{X/k}(t) = Z_{X/K}\big(t^{[K:k]}\big).
\]
holds by definition and that the order of the pole of $Z_{X/K}(t)$ at $t=1$ is equal to $1$.
The latter assertion can be proved as follows: by the famous functional equation of the zeta function, it results that the order at $t=1$ equals to the one at $t=q^{d}$, which is already shown to be $1$ in \cite{Se}.
The details are left to the reader.
\end{proof}

\begin{rem}\label{extension_to_motives}
\upshape
The statement of Proposition \ref{proper_smooth_case} holds not only for proper smooth schemes but also for effective Chow motives i.e. direct summands of proper smooth schemes in the additive category $\mathrm{Chow}^{\mathrm{eff}}(k)$.

To be precise, fix a prime number $l \neq ch(k)$, and define zeta functions of effective Chow motives by Grothendieck's determinant formula (note that $l$-adic etale cohomology can be defined for motives, and that arithmetic Frobenius action is compatible with projectors).
This is a rational function with $\mathbb{Q}_{l}$ coefficients.
On the other hand, the regulator maps are also compatible with projectors.
Thus, the both sides of the equation in Proposition are extended for effective Chow motives.

Let $M$ be an effective Chow motive.
To prove the equation for $M$, we can assume that $M$ is a direct summand of proper smooth geometrically irreducible scheme over $k$.
Then, observe that in this case the equation for $X$ was obtained using Milne's formula and cycle isomorphism between higher Chow groups and etale cohomology.
Since cycle isomorphism is compatible with projectors, it is enough to show Milne's formula for effective Chow motives, which can be done using the fact that it comes from Grothendieck's deteminant formula.
\end{rem}

Before stating the first main result of this paper, let us observe a basic fact of homological algebra.

\begin{lem}\label{morphism_of_complexes}
Let 
\[
C_{\ast} : \cdots \to C_{N} \stackrel{d_{N}}{\to} C_{N-1} \stackrel{d_{N-1}}{\to} C_{N-2} \to \cdots
\]
and
\[
C^{\prime}_{\ast} : \cdots \to C^{\prime}_{N} \stackrel{d^{\prime}_{N}}{\to} C^{\prime}_{N-1} \stackrel{d^{\prime}_{N-1}}{\to} C^{\prime}_{N-2} \to \cdots
\]
be two homological complexes of abelian groups,
and $f_{\ast} : C_{\ast} \to C^{\prime}_{\ast}$ be a morphism of complexes.
Assume one of the following conditions:

(i) $C_{i}$ and $C^{\prime}_{i}$ are objects of $\mathrm{FQ}$ for each $i$.
and that all but finitely many $C_{i}, C^{\prime}_{i}$ are uniquely divisible,

(ii) all but finitely many $C_{i}, C^{\prime}_{i}$ are zero, and kernel and cokernel of $f_{i}$ are finite groups for any $i$.

Then, the kernel and cokernel of the map $H_{i}(f_{\ast}) : H_{i}(C_{\ast}) \to H_{i}(C^{\prime}_{\ast})$ induced on the $i$-th homology groups have finite torsion subgroups for each $i$, and the following equation holds:
\[
\prod_{i}\chi (H_{i}(f_{\ast}))_{\mathrm{tor}}^{(-1)^{i}} = \prod_{i}\chi (f_{i})_{\mathrm{tor}}^{(-1)^{i}}.
\]
\end{lem}

\begin{proof}
Noting that the functor $(-)_{\mathrm{tor}} : \mathrm{FQ} \to \text{(finite groups)}$ is exact, the case (ii) is reduced to the case (i), whcih can be proved by induction on the length of the complexes and the snake lemma.
\end{proof}

\begin{thm}\label{boundedness}
For any quasi-projective scheme $X$ over a perfect field $k$ of exponential characteristic $p$,
the weight complex
$W(X)[\frac{1}{p}] \in K^{+}(\mathrm{Chow}^{\mathrm{eff}}(k)[\frac{1}{p}])$
has a representation of the form
\[
0 \to X_{d} \to X_{d-1} \to \cdots \to X_{1} \to X_{0} \to 0
\]
where $d=dim(X)$ and $X_{a}$ is a Chow motive for each $a$.
\end{thm} 

\begin{rem}
\upshape
If resolution of singularities holds over $k$,
then the theorem is valid for any varieties over $k$,
since $W(X)$ can be essentially calculated by using irreducible components of simple normal crossing divisors.
For $p>0$, we prove the theorem by comparing our weight complex with the one defined by Bondarko, which is known to be bounded in the homotopy category of (effective) Chow motives.
\end{rem}

\begin{proof}
Let $t : \mathrm{DM}_{\mathrm{gm}}^{\mathrm{eff}}(k)[1/p] \to \mathcal{H}^{b}(\mathrm{Chow}^{\mathrm{eff}}(k)[1/p])$ be the triangulated functor called "weight complex functor" defined by Bondarko (see \cite{Bon}, \cite{Bon2}).
According to Kelly's result on Voevodsky's motives over a field of positive characteristic, the functor $M^{c}(-) : (\mathrm{Sch}_{k})^{\mathrm{prop}} \to \mathrm{DM}_{-}^{\mathrm{eff}}(k)[1/p]$ which associates every variety over $k$ its motive with compact support factors through the full subcategory $\mathrm{DM}_{\mathrm{gm}}^{\mathrm{eff}}(k)[1/p]$ of geometric motives with $p$ inverted (see \cite[Lemma 5.5.2, 5.5.6]{Kel}).
Here, $(\mathrm{Sch}_{k})^{\mathrm{prop}}$ is the subcategory of $\mathrm{Sch}_{k}$ whose objects are the same and morphisms are proper morphisms.
Composing these functors, we get a covariant functor $t \circ M^{c}(-) : (\mathrm{Sch}_{k})^{\mathrm{prop}} \to \mathcal{H}^{b}(\mathrm{Chow}^{\mathrm{eff}}(k)[1/p])$.

It is enough to show that for any quasi-projective variety $X$ over $k$, the two objects $W(X)$ and $t \circ M^{c}(X)$ in $\mathcal{H}^{-}(\mathrm{Chow}^{\mathrm{eff}}(k)[1/p])$ are isomorphic.
Indeed, $t \circ M^{c}(X)$ is bounded.
One can prove the assertion on the length of the weight complex as follows: by induction on dimension and the localization distinguished triangle of motives with compact support, the assertion is reduced to the case that $X$ is smooth.
In this case, Poincare duality $M^{c}(X)^{\ast} = M(X)(-d)[-2d]$ (\cite{Kel}) and the results established by Bondarko (Theorem 2.2.1 (3),  Proposition 3.2.1 (2), Proposition 3.5.1 in \cite{Bon2}) show the assertion.
 
Now, to construct an isomorphism $W(X) \simeq t \circ M^{c}(X)$, take a compactification $X \to \overline{X}$ with complement $Y = \overline{X} \setminus X$.
Then there exists a distinguished triangle in $\mathrm{DM}_{\mathrm{gm}}^{\mathrm{eff}}(k)[1/p]$ of the form
\[
M^{c}(Y) \to M^{c}(\overline{X}) \to M^{c}(X) \stackrel{+}{\to} 
\]
which induces the one in $\mathcal{H}^{b}(\mathrm{Chow}^{\mathrm{eff}}(k)[1/p])$ of the form
\[
t \circ M^{c}(Y) \to t \circ M^{c}(\overline{X}) \to t \circ M^{c}(X) \stackrel{+}{\to}.
\]

By induction on dimension, and considering that our weight complex $W$ has an analogous triangle and that $dim(Y) < dim(X)$,  we are reduced to show the following assertion:
for any projective variety $Z$, there exists an isomorphism $W(Z) \simeq t \circ M^{c}(Z)$ which is compatible with closed immersions.

Let $Z$ be a projective variety over $k$, and $\tilde{Z}_{\cdot} \to Z$ be a projective smooth $\Lambda$-admissible hyperenvelope of $Z$.
By construction of $t$, we have $t(M(\tilde{Z}_{\cdot})) \simeq \tilde{Z}_{\cdot}$ in $\mathcal{H}^{b}(\mathrm{Chow}^{\mathrm{eff}}(k)[1/p])$ since each $\tilde{Z}_{i}$ is projective smooth.
We have only to show that $t(M^{c}(Z)) = t(M(Z)) \simeq t(M(\tilde{Z}_{\cdot})) (\simeq \tilde{Z}_{\cdot})$

Let us prove that the morphism $M(\tilde{Z}_{\cdot}) \to M(Z)$ is an isomorphism in $\mathrm{DM}_{-}^{\mathrm{eff}}(k)[1/p]$.
This is equivalent to say that it is an isomorphism in $\mathrm{DM}_{-}^{\mathrm{eff}}(k)_{(l)}$ for every prime number $l \neq p$.
In the following, we fix such a prime number $l$.
By Kelly's resolution theorem (see \cite[Theorem 5.3.1]{Kel}), it suffices to show that the morphism of complexes of $l$dh sheaves with transfers $L(\tilde{Z}_{\cdot})_{ldh} \to L(Z)_{ldh}$ is a quasi isomorphism.
Since $\tilde{Z}_{\cdot} \to Z$ is a $\Lambda$-admissible hyperenvelope, it is an $l^{\prime}$-decomposed hypercovering by definition, hence an $l$dh hypercovering in Kelly's terminology.
The following lemma concludes the desired statement.
\end{proof}

\begin{lem}
Let $k$ be as above, $l \neq p$ be a prime number and $\tilde{X}_{\cdot} \to X$ be an $l$dh hypercovering of a variety $X$ over $k$.
Then, the morphism $L(\tilde{X}_{\cdot})_{ldh} \to L(X)_{ldh}$ of complexes of $l$dh sheaves with transfers is a quasi isomorphism.
\end{lem}

\begin{proof}
By Yoneda's lemma, the assertion is equivalent to saying that the induced map $\mathrm{Hom}_{D_{\mathrm{tr}}}(L(X)_{ldh}, K) \to \mathrm{Hom}_{D_{\mathrm{tr}}}(L(\tilde{X}_{\cdot})_{ldh}, K)$ of hom sets of the derived category $D_{\mathrm{tr}}=D^{-}(\mathrm{Shv}_{ldh}(\mathrm{Cor}_{k}))$ is bijective for any bounded above complex $K$ of $l$dh sheaves with transfers on the category $\mathrm{Sch}_{k}$ of separated schemes of finite type over $k$.
By standard arguments on way-out functors and by the finite dimensionality of $l$dh cohomology (for sheaves with transfers, see \cite[Theorem 3.4.17]{Kel}), we can assume that $K=F[i]$, where $F$ is an $l$dh sheaf with transfers and $i$ is an integer.
Since $\tilde{X}_{\cdot}$ is an $l$dh hypercover of $X$, we have a quasi isomorphism $\mathbb{Z}(\tilde{X}_{\cdot})_{ldh} \to \mathbb{Z}(X)_{ldh}$ of $l$dh sheaves on schemes (without transfers), which induces
\begin{align*}
&\mathrm{Hom}_{D_{\mathrm{tr}}}(L(X)_{ldh}, F[i]) \\
&\stackrel{(1)}{=}H_{ldh}^{i}(X, F)
\stackrel{(2)}{=}\mathrm{Hom}_{D}(\mathbb{Z}(X)_{ldh}, F[i]) \\
&\stackrel{(3)}{=}\mathrm{Hom}_{D}(\mathbb{Z}(\tilde{X}_{\cdot})_{ldh}, F[i])
\stackrel{(4)}{=}\mathrm{Hom}_{D^{+}}(\tau_{\geq i}(\mathbb{Z}(\tilde{X}_{\cdot})_{ldh}), F[i]) \\
&\stackrel{(5)}{=}\mathrm{Hom}_{D^{+}}((\tau_{\geq i}(\mathbb{Z}(\tilde{X}_{\cdot})))_{ldh}, F[i]) 
\stackrel{(6)}{=}\mathrm{Hom}_{D_{\mathrm{pre}}^{+}}(\tau_{\geq i}(\mathbb{Z}(\tilde{X}_{\cdot})), D^{+}(\iota)(F[i])) \\
&\stackrel{(7)}{=}\mathrm{Hom}_{D_{\mathrm{pre}}^{+}}(\tau_{\geq i}(\mathbb{Z}(\tilde{X}_{\cdot})), I^{\cdot}[i]) 
\stackrel{(8)}{=}\mathrm{Hom}_{D_{\mathrm{pre}}}(\mathbb{Z}(\tilde{X}_{\cdot}), I^{\cdot}[i]) \\
&\stackrel{(9)}{=}\mathrm{Hom}_{\mathcal{H} (\mathrm{PSh}(\mathrm{Sch}_{k}))} (\mathbb{Z}(\tilde{X}_{\cdot}), I^{\cdot}[i]) 
\stackrel{(10)}{=}\mathrm{Hom}_{\mathcal{H} (\mathrm{PSh}(\mathrm{Cor}_{k}))} (L(\tilde{X}_{\cdot}), I^{\cdot}[i]) \\
&\stackrel{(11)}{=}\mathrm{Hom}_{D_{\mathrm{tr},\mathrm{pre}}}(L(\tilde{X}_{\cdot}), I^{\cdot}[i]) 
\stackrel{(12)}{=}\mathrm{Hom}_{D^{+}_{\mathrm{tr},\mathrm{pre}}} (\tau_{\geq i} L(\tilde{X}_{\cdot}), I^{\cdot}[i]) \\
&\stackrel{(13)}{=}\mathrm{Hom}_{D^{+}_{\mathrm{tr}}}((\tau_{\geq i} L(\tilde{X}_{\cdot}))_{ldh}, F[i]) 
\stackrel{(14)}{=}\mathrm{Hom}_{D^{+}_{\mathrm{tr}}}(\tau_{\geq i} (L(\tilde{X}_{\cdot})_{ldh}), F[i]) \\
&\stackrel{(15)}{=}\mathrm{Hom}_{D_{\mathrm{tr}}}(L(\tilde{X}_{\cdot})_{ldh}, F[i])
\end{align*}
where
\[
D=D(\mathrm{Shv}_{ldh}(\mathrm{Sch}_{k})), D_{\mathrm{tr},\mathrm{pre}}=D(\mathrm{PSh}(\mathrm{Cor}_{k})), D_{\mathrm{pre}}=D(\mathrm{PSh}(\mathrm{Sch}_{k})),
\]
$\iota : \mathrm{Shv}_{ldh}(\mathrm{Sch}_{k}) \to \mathrm{PSh}(\mathrm{Sch}_{k})$ is the inclusion functor, 
and $F \to I^{\cdot}$ is an injective resolution of $F$ as an $l$dh sheaf with transfers.
Note that any injective object in $\mathrm{Shv}_{ldh}(\mathrm{Cor}_{k})$ is an acyclic object in $\mathrm{Shv}_{ldh}(\mathrm{Sch}_{k})$, we can calculate $l$dh cohomology of $F$ (hence the derived functor $D^{+}(\iota)$) using $I^{\cdot}$, which proves the equation (1).
(2) is obvious.
(3) follows from $\mathbb{Z}(\tilde{X}_{\cdot})_{ldh} \simeq \mathbb{Z}(X)_{ldh}$.
(4), (8), (12) and (15) follow from the adjunction with respect to $t$-structures.
(5) and (14) are exactness of sheafification.
(6) and (13) are the adjuction of derived functors induced by sheafification and inclusion functors (note that these functors are defined at least on the derived categories of bounded below complexes).
Here, we denote by $F_{ldh}$ the $l$dh sheafification of a presheaf with transfers $F$ restrected to $\mathrm{Sch}_{k}$.
The $l$dh sheaf $F_{ldh}$ has a canonical structure of transfers, and one can prove that the functor $(-)_{ldh} : \mathrm{PSh}(\mathrm{Cor}_{k}) \to \mathrm{Shv}_{ldh}(\mathrm{Cor}_{k})$ is left adjoint to the inclusion functor (see \cite[Corollary 3.4.12, Theorem3.4.13]{Kel}). 
(7) follows from a similar argument as in (1).
(9) and (11) are consequences of the projectiveness of $\mathbb{Z}(\tilde{X}_{\cdot})$ and $L(\tilde{X}_{\cdot})$ as presheaves.
Finally, (10) is a direct computation.
\end{proof}

\begin{thm}\label{main_thm_bounded}
Let $k$ be a finite field $\mathbb{F}_{q}$ with characteristic $p$, and $X$ be a quasi-projective scheme of dimension $d$ over $k$.
Suppose that the condition $( \spadesuit )_{X/k}$ holds.
Then, $\chi (\mathrm{Reg}_{i}(X))$ is well-defined as a positive rational number for each $i$, equals to $1$  for each $i>2d$, and the following equation holds up to sign and a power of $p$:
\[
\zeta_{X}(0)^{\ast} = \prod_{i=0}^{2d} \chi (\mathrm{Reg}_{i}(X))^{(-1)^{i+1}}.
\]
Moreover, if resolution of singularities exists for any  variety of dimension $\leq d$ over $k$, then the equation holds only up to sign.
\end{thm}

\begin{rem}\label{parshin_bijectivity}
Assume that one of the equivalent conditions of Corollary \ref{parshin} holds for any smooth projective schemes.
Then, the regulator map $\mathrm{Reg}_{i}(X)_{\mathbb{Q}}$ tensored by $\mathbb{Q}$ is bijective for any quasi-projective scheme $X$ (see the proof below).
\end{rem}

\begin{proof}

(I) Firstly, we will show that for any quasi-projective $X$ of dimension $d$, $\mathrm{CH}_{0}(X, i)$ is uniquely divisible if $i>2d$ and $H_{j}^{W}(X)$ is zero if $j>d$.

By induction on dimension and localization exact sequences, we can assume that $X$ is smooth.
In this case, fix $i > 2\dim(U)$, $j>d$ and a prime number $l \neq p$.
It suffices to show that $\mathrm{CH}_{0}(U,i) \otimes \mathbb{Z}_{(l)}$ is uniquely divisible and $H^{W}_{j}(U) \otimes \mathbb{Z}_{(l)}$ is zero.

In the case that $U$ has a smooth compactification with the complement simple normal crossing divisor, this can be observed easily considering Remark \ref{rem_sncd}.

In general case, take an open immersion $U \to X$ with $X$ projective.
Regard the closed complement $Y=X \setminus U$ as a closed subscheme of $X$ with the reduced structure.
Take an $l^{\prime}$-alteration $\pi : X^{\prime} \to X$ such that $X^{\prime}$ is smooth projective and the fiber of $Y$ is a simple normal crossing divisor on $X^{\prime}$ (\cite{Il}).

Since the inverse image $U^{\prime} = \pi^{-1}(U)$ is a smooth scheme with good compactification, $\mathrm{CH}_{0}(U^{\prime},i) \otimes \mathbb{Z}_{(l)}$ is uniquely divisible and $H_{j}(U^{\prime}) \otimes \mathbb{Z}_{(l)}$ is zero.
By induction on dimension and by localization long exact sequences, we can show the following statement: for any open dence immersion $W \to V$ of smooth schemes,$\mathrm{CH}_{0}(W,i)\otimes \mathbb{Z}_{(l)}$ is uniquely divisible (resp. $H_{j}(U^{\prime}) \otimes \mathbb{Z}_{(l)}$ is zero) if and only if $\mathrm{CH}_{0}(V,i)\otimes \mathbb{Z}_{(l)}$ is uniquely divisible (resp. $H_{j}(V^{\prime}) \otimes \mathbb{Z}_{(l)}$  is zero).

Hence, replacing $U$ by an open dence subscheme, we can assume that $g=\pi|_{U} : U^{\prime} \to U$ is a finite flat morphsim of degree prime to $l$ and that $\mathrm{CH}_{0}(U^{\prime},i) \otimes \mathbb{Z}_{(l)}$ is uniquely divisible and $H_{j}^{W}(U^{\prime}) \otimes \mathbb{Z}_{(l)}$ is zero.

The fact that higher Chow groups satisfy the condition (v) of Theorem \ref{thm:descent2}(b) shows that $g_{\ast}g^{\ast} : \mathrm{CH}_{0}(U,i) \otimes \mathbb{Z}_{(l)} \to \mathrm{CH}_{0}(U^{\prime},i) \otimes \mathbb{Z}_{(l)} \to \mathrm{CH}_{0}(U,i) \otimes \mathbb{Z}_{(l)}$ is just the multiplication by the degree of $g$, hence injective.
This proves that $\mathrm{CH}_{0}(U,i) \otimes \mathbb{Z}_{(l)}$ is uniquely divisible.

In order to show that $H_{j}^{W}(U)$ is zero, consider another homology theory $H_{\ast}^{W}(-,\mathbb{Q}_{l}/\mathbb{Z}_{l})$ obtained by replacing $\Lambda$ by $\mathbb{Q}_{l}/\mathbb{Z}_{l}$ in the definition of weight homology theory.
Since weight homology is the homology of a complex of free $\Lambda$-modules, there exists a short exact sequence
\[
0 \to H_{j}^{W}(U) \otimes \mathbb{Q}_{l}/\mathbb{Z}_{l}
\to H_{j}^{W}(U,\mathbb{Q}_{l}/\mathbb{Z}_{l})
\to H_{j-1}^{W}(U)\{ l \} \to 0.
\]
Weight homology group is finitely generated by definition.
So, we have only to show that $H_{j}^{W}(U,\mathbb{Q}_{l}/\mathbb{Z}_{l})$ vanishes.
By \cite[Lemma 9.5]{KeS2}, we have $H_{j}^{W}(U,\mathbb{Q}_{l}/\mathbb{Z}_{l}) \simeq KH_{j}(U,\mathbb{Q}_{l}/\mathbb{Z}_{l})$, where $KH_{j}$ is Kato homology theory (see \cite{KeS1} for the definition).

Since Kato homology with torsion coefficients satisfy the condition (v) of Theorem \ref{thm:descent2}(b) (\cite[Section 4]{KeS1}), we are able to show the vanishing of $KH_{j}(U,\mathbb{Q}_{l}/\mathbb{Z}_{l})$ by the same argument as above.

\ 

(II) Now, let us prove the equation of the theorem.
For any quasi-projective scheme $Z$ over $k$,
let $K_{\ast}(Z)$ denote the kernel complex of the natural surjection $z_{0}(Z, \ast)_{\Lambda} \to \mathrm{CH}_{0}(Z)_{\Lambda}[0]$ i.e.
\[
0 \to K_{\ast}(Z) \to z_{0}(Z, \ast)_{\Lambda} \to \mathrm{CH}_{0}(Z)_{\Lambda}[0] \to 0.
\]

Let $X$ be a quasi-projective scheme of dimension $d$ over $k$.
By Theorem \ref{boundedness}, the weight complex $W(X)$ of $X$ is homotopically equivalent to a bounded complex of effective Chow motives $(X_{a})_{0 \leq a \leq d}$.
The above exact sequence induces an exact sequence of total complexes
\[
0 \to \mathrm{tot}(K_{\ast}(X_{\bullet})) \to \mathrm{tot}(z_{0}(X_{\bullet}, \ast)_{\Lambda}) \stackrel{\alpha}{\to} \mathrm{tot}(\mathrm{CH}_{0}(X_{\bullet}, \ast)_{\Lambda}[0]) \to 0.
\]
In addition, consider a morphism of complexes induced by the degree map:
\[
\mathrm{tot}(\mathrm{CH}_{0}(X_{\bullet}, \ast)_{\Lambda}[0]) \stackrel{\beta}{\to} \mathrm{tot}(\Lambda^{\pi{X_{\bullet}}}[0]).
\]
Note that the composition of maps $\beta \circ \alpha$ induces the regulator maps by definition.
Thus, by the assumption of surjectivity and the long exact sequence which comes from the short exact sequence of complexes, the assertion in the theorem is reduced to the following two statements: (i) the homology groups $H_{i}(\mathrm{tot}(K_{\ast}(X_{\bullet})))$ are objects in $\mathrm{FQ}$ (including the case $i=0$), (ii) the equation
\[
\zeta_{X}(0)^{\ast}
=  \prod_{i=0}^{2d} \chi(H_{i}(\alpha))^{(-1)^{i+1}} \cdot \chi(H_{i}(\beta))^{(-1)^{i+1}}
\]
holds (up to sign and $p$-power), and 
(iii) for any homomorphisms of abelian groups $f : A \to B$ and $g: B \to C$ whose kernels and cokernels are objects of $\mathrm{FQ}$, the same assertion is true for the composition $gf$,  and the equation $\chi(gf) = \chi (g) \cdot \chi (f)$ holds.

The statement (i) is valid since $\mathrm{tot}(K_{\ast}(X_{\bullet}))$ is a complex of objects in $\mathrm{FQ}$ as we have seen in Proposition \ref{parshin}.

The statement (iii) is an easy consequence of the following exact sequence of abelian groups:
\[
0 \to \mathrm{Ker}(f) \to \mathrm{Ker}(gf) \stackrel{f}{\to} \mathrm{Ker}(g) \to \mathrm{Cok}(f) \stackrel{g}{\to} \mathrm{Cok}(gf) \to \mathrm{Cok}(g) \to 0.
\]

Let us prove the statement (ii).
By (i), Lemma \ref{morphism_of_complexes} and Theorem \ref{boundedness}, the right hand side of the equation in (ii) is calculated as follows:
\[
\prod_{i=0}^{2d} \chi(H_{i}(\alpha))^{(-1)^{i+1}}
=
\prod_{b>0, a}  \# \mathrm{CH}_{0}(X_{a}, b)_{\Lambda, tor} ^{(-1)^{a+b}},
\]
\[
\prod_{i=0}^{2d} \chi(\beta)^{(-1)^{i+1}}
=
\prod_{a} \chi(\mathrm{deg} : \mathrm{CH}_{0}(X_{a})_{\Lambda} \to \Lambda^{\pi_{0}(X_{a})})^{(-1)^{a}}.
\]
Hence, the equation in (ii) immediately follows from Proposition \ref{proper_smooth_case},  Remark \ref{extension_to_motives} and the following lemma.
\end{proof}

\begin{lem}
\[
\zeta_{X}(0)^{\ast} = \prod_{a=0}^{d} \zeta_{X_{a}}(0)^{\ast} 
\]
(see Remark \ref{extension_to_motives} for the precise meaning of the right hand side)
\end{lem}

\begin{proof}
Without loss of generality, one can assume that $X$ is projective.
Indeed, the weight complex of a quasi-projective scheme can be represented as the cone of a morphism between weight complexes of projective schemes (i.e. a compactification and the reduced complement).
In this case, the weight complex of $X$ is a complex associated to a hyperenvelope of $X$, which is in particular a proper hypercovering of $X$.
Then, the desired statement is an easy consequence of Grothendieck's determinant formula of zeta functions and the cohomological descent of \'{e}tale cohomology.
\end{proof}

\begin{cor}\label{s<0_bounded}
Let $k$ be a finite field $\mathbb{F}_{q}$, $p=ch(k)$ and $X$ be a quasi-projective scheme of dimension $d$ over $k$.
Then, for any negative integer $r$, $\chi (\mathrm{Reg}_{i}(X \times \mathbb{A}^{-r}))$ is well-defined for each $i$, equals to $1$  for each $i>2(d-r)$, and the following equations hold up to sign and a power of $p$:
\[
\zeta_{X}(r)^{\ast} = \prod_{i=0}^{2(d-r)} \chi (\mathrm{Reg}_{i}(X \times \mathbb{A}^{-r}))^{(-1)^{i+1}}
=\prod_{i=0}^{2(d-r)} (\# \mathrm{CH}_{r}(X,i)_{\mathrm{tor}})^{(-1)^{i}}.
\]
Moreover, if resolution of singularities exists for any  variety of dimension $\leq d$ over $k$, then the equation holds only up to sign.
\end{cor}

\begin{proof}
This follows from the homotopy invariance of higher Chow groups, the equation $\zeta_{X \times \mathbb{A}^{1}}(a)^{\ast} = \zeta_{X}(a-1)^{\ast}$ and the vanishing $H^{W}_{i}(X \times \mathbb{A}^{1}) = 0$ for any integer $i$.
The vanishing of weight homology is an easy consequence of the localization exact sequence and the identification $\mathbb{Z}^{\pi_{0}(X)} = \mathbb{Z}^{\pi_{0}(X \times \mathbb{P}^{1})}$.

Note that the surjectivity assumption as in Theorem \ref{main_thm_bounded} is always satisfied since weight homology groups are all zero.
\end{proof}

According to \cite{CP1} and \cite{CP2}, resolution of singularities exists for any (reduced) variety of dimension $\leq 3$, which shows the following corollary.

\begin{cor}
For any quasi-projective scheme $X$ satisfying $\dim(X) \leq 3$,
the equation in Theorem \ref{main_thm_bounded} holds only up to sign without the assumption of resolution of singularities.
\end{cor}

\end{document}